\documentclass{amsart}
\usepackage[utf8]{inputenc}
\usepackage{amssymb}
\usepackage{hyperref}
\usepackage[final]{showkeys}

\input xy
\xyoption{all}

\theoremstyle{definition}
\newtheorem{mydef}{Definition}[section]
\newtheorem{lem}[mydef]{Lemma}
\newtheorem{thm}[mydef]{Theorem}
\newtheorem{cor}[mydef]{Corollary}

\newtheorem{question}[mydef]{Question}
\newtheorem{hypothesis}[mydef]{Hypothesis}

\newtheorem{defin}[mydef]{Definition}
\newtheorem{example}[mydef]{Example}
\newtheorem{remark}[mydef]{Remark}
\newtheorem{notation}[mydef]{Notation}
\newtheorem{fact}[mydef]{Fact}

\def\x{\mathbf{x}}
\def\y{\mathbf{y}}

\newcommand{\sq}[2]{\sideset{^{#1}}{}{\operatorname{#2}}}
\newcommand{\ba}{\bar{a}}

\newcommand{\bx}{\bar{x}}
\newcommand{\by}{\bar{y}}

\newcommand{\sea}{\mathfrak{C}}

\newcommand{\seq}[1]{\langle #1 \rangle}
\newcommand{\rest}{\upharpoonright}

\newcommand{\is}{\mathfrak{i}}

\newcommand{\insp}[1]{\is_{#1\text{-ns}}}

 \newcommand{\LS}{\operatorname{LS}}
\newcommand{\D}{\operatorname{D}}
 \newcommand{\Mod}{\operatorname{Mod}}
\newcommand{\EM}{\operatorname{EM}}
\newcommand{\Av}{\operatorname{Av}}
\newcommand{\K}{\mathbf{K}}

 \renewcommand{\phi}{\varphi}
\newcommand{\initial}\lessdot
 \newcommand{\infinity}{\infty}
\def\?{?\vadjust
{\vbox to 0pt{\vskip-7pt\hbox to 1.1\hsize{\hfill\huge ?!}}}}

\newbox\noforkbox \newdimen\forklinewidth
\setbox0\hbox{$\textstyle\smile$}
\setbox1\hbox to \wd0{\hfil\vrule width \forklinewidth depth-2pt
 height 10pt \hfil}
\setbox\noforkbox\hbox{\lower 2pt\box1\lower 2pt\box0\relax}
\def\unionstick{\mathop{\copy\noforkbox}\limits}

\def\nonfork_#1{\unionstick_{\textstyle #1}}

\setbox0\hbox{$\textstyle\smile$}
\setbox1\hbox to \wd0{\hfil{\sl /\/}\hfil}
\setbox2\hbox to \wd0{\hfil\vrule height 10pt depth -2pt width
               \forklinewidth\hfil}
\newbox\doesforkbox
\setbox\doesforkbox\hbox{\lower 2pt\box1 \lower 2pt\box2\lower2pt\box0\relax}
\def\nunionstick{\mathop{\copy\doesforkbox}\limits}

\def\fork_#1{\nunionstick_{\textstyle #1}}

\newcommand{\leap}[1]{\le_{#1}}
\newcommand{\ltap}[1]{<_{#1}}

\newcommand{\geap}[1]{\ge_{#1}}

\newcommand{\lta}{\ltap{\K}}
\newcommand{\lea}{\leap{\K}}

\newcommand{\gea}{\geap{\K}}

\def\lee{\preceq}

\newbox\noforkbox \newdimen\forklinewidth
\forklinewidth=0.3pt \setbox0\hbox{$\textstyle\smile$}
\setbox1\hbox to \wd0{\hfil\vrule width \forklinewidth depth-2pt
 height 10pt \hfil}
\wd1=0 cm \setbox\noforkbox\hbox{\lower 2pt\box1\lower
2pt\box0\relax}
\def\unionstick{\mathop{\copy\noforkbox}\limits}

\setbox0\hbox{$\textstyle\smile$}
\setbox1\hbox to \wd0{\hfil{\sl /\/}\hfil} \setbox2\hbox to
\wd0{\hfil\vrule height 10pt depth -2pt width
               \forklinewidth\hfil}
\newbox\doesforkbox
\setbox\doesforkbox\hbox{\lower 0pt\box1 \lower
2pt\box2\lower2pt\box0\relax}

\def\nunionstick{\mathop{\copy\doesforkbox}\limits}

\def\1nf{\unionstick^{(1)}}
\def\2nf{\unionstick^{(2)}}

\newcommand{\gtp}{\text{gtp}}

\newcommand{\gS}{\text{gS}}

\newcommand{\BI}{\mathbf{I}}
\newcommand{\BJ}{\mathbf{J}}
\newcommand{\clc}[1]{\kappa_{#1}}

\newcommand{\hanfe}[1]{\beth_{\left(2^{#1}\right)^+}}

\newcommand{\Ll}{\mathbb{L}}

\newcommand{\Ksatp}[1]{\K^{#1\text{-sat}}}
\newcommand{\Ksatpnobf}[1]{K^{#1\text{-sat}}}

\newcommand{\Kmhp}[1]{\K^{#1\text{-mh}}}
\newcommand{\Kmh}[1]{\Kmhp{\lambda}}

\title[Definitions of superstability in tame AECs]{Equivalent definitions of superstability in tame abstract elementary classes}
\author{Rami Grossberg}
\email{rami@cmu.edu}
\urladdr{http://math.cmu.edu/\textasciitilde rami}
\address{Department of Mathematical Sciences, Carnegie Mellon University, Pittsburgh, Pennsylvania, USA}

\author{Sebastien Vasey}
\thanks{This material is based upon work done while the second author was supported by the Swiss National Science Foundation under Grant No.\ 155136.}
\email{sebv@cmu.edu}
\urladdr{http://math.cmu.edu/\textasciitilde svasey/}
\address{Department of Mathematical Sciences, Carnegie Mellon University, Pittsburgh, Pennsylvania, USA}

\date{\today }

\begin{document}

\begin{abstract}
  In the context of abstract elementary classes (AECs) with a monster model, several possible definitions of superstability have appeared in the literature. Among them are no long splitting chains, uniqueness of limit models, and solvability. Under the assumption that the class is tame and stable, we show that (asymptotically) no long splitting chains implies solvability and uniqueness of limit models implies no long splitting chains. Using known implications, we can then conclude that all the previously-mentioned definitions (and more) are equivalent:  

\begin{cor}\label{abstract-thm}
  Let $\K$ be a tame AEC with a monster model. Assume that $\K$ is stable in a proper class of cardinals. The following are equivalent:

  \begin{enumerate}
    \item\label{abstract-1} For all high-enough $\lambda$, $\K$ has no long splitting chains.
    \item\label{abstract-2} For all high-enough $\lambda$, there exists a good $\lambda$-frame on a skeleton of $\K_\lambda$.
    \item\label{abstract-3} For all high-enough $\lambda$, $\K$ has a unique limit model of cardinality $\lambda$.
    \item\label{abstract-4} For all high-enough $\lambda$, $\K$ has a superlimit model of cardinality $\lambda$.
    \item\label{abstract-5} For all high-enough $\lambda$, the union of any increasing chain of $\lambda$-saturated models is $\lambda$-saturated.
    \item\label{abstract-6} There exists $\mu$ such that for all high-enough $\lambda$, $\K$ is $(\lambda, \mu)$-solvable.
  \end{enumerate}
\end{cor}

This gives evidence that there is a clear notion of superstability in the framework of tame AECs with a monster model.  

\end{abstract}

\maketitle

\tableofcontents

\section{Introduction}

In the context of classification theory for abstract elementary classes (AECs), a notion analogous to the first-order notion of \emph{stability} exists: let us say that an AEC $\K$ is \emph{stable in $\lambda$} if $\K$ has at most $\lambda$-many Galois types over every model of cardinality $\lambda$ (a justification for this definition is Fact \ref{stab-spectrum}, showing that it is equivalent, under tameness, to failure of the order property). However it has been unclear what a parallel notion to superstability might be.  Recall that for first-order theories we have:

\begin{fact}\label{fo-superstab} Let $T$ be a first-order complete theory. The following are equivalent:
\begin{enumerate}

\item $T$ is stable in every cardinal $\lambda \ge 2^{|T|}$.

\item For all infinite cardinals $\lambda$, the union of an increasing chain of $\lambda$-saturated models is $\lambda$-saturated.

\item $\kappa(T)=\aleph_0$ and $T$ is stable.

\item $T$ has a saturated model of cardinality $\lambda$ for every $\lambda\geq 2^{|T|}$.

\item $T$ is stable and $\D^n[\bx=\bx,L (T),\infinity]<\infinity$.

\item There does not exists a set of formulas $\Phi=\{\phi_n(\bx;\by_n)\mid n<\omega\}$ such that $\Phi$ can be used to code the structure $(\omega^{\leq\omega},<,<_{lex})$

%\item
\end{enumerate}

\end{fact}
$(1) \implies (2)$ and $(1) \iff (\ell)$ for $\ell \in \{3, 4, 5, 6\}$ all appear in Shelah's book \cite{shelahfobook}. Albert and Grossberg \cite[13.2]{agchains} established $(2)\implies (6)$.

In the last 30 years, in the context of classification theory for non elementary classes,  several notions that generalize that of first-order superstablity have been considered.  See papers by Grossberg, Shelah, VanDieren, Vasey and Villaveces:  \cite{grsh238,gr88}, \cite{sh394}, \cite{shvi635}, \cite{vandierennomax, nomaxerrata}, \cite{gvv-mlq}, \cite{ss-tame-jsl, indep-aec-apal}. Reasons for developping a superstability theory in the non-elementary setup include the aesthetic appeal (guided by motivation from the first-order case) and recent applications such as Shelah's eventual categoricity conjecture in universal classes \cite{ap-universal-v10, categ-universal-2-v3-toappear} or the fact that (in an AEC with a monster model) the model in a categoricity cardinal is saturated \cite{categ-saturated-v2}.

In \cite[p.~267]{sh394} Shelah states that part of the program of classification theory for AECs is to show that all the various notions of first-order saturation (limit, superlimit, or model-homogeneous, see Section \ref{sat-def-subsec}) are equivalent under the assumption of  superstablity. A possible definition of superstability is \emph{solvability} (see Definition \ref{solvability-def}), which appears in the introduction to \cite{shelahaecbook} and is hailed as a true counterpart to first-order superstability. Full justification is delayed to \cite{sh842} but \cite[Chapter IV]{shelahaecbook} already uses it. Other definitions of superstability analogous to the ones in Fact \ref{fo-superstab} can also be formulated. The main result of this paper is to show that, at least in tame AECs with a monster model, several definitions of superstability that previously appeared in the literature are equivalent (see the preliminaries for precise definitions of some of the concepts appearing below). Many of the implications have already been proven in earlier papers, but here we complete the loop by proving two theorems. Before stating them, some notation will be helpful:

\begin{notation}[4.24(5) in \cite{baldwinbook09}]\label{hanf-notation}
  Given a fixed AEC $\K$, set $H_1 := \hanfe{\LS (\K)}$.
\end{notation}

\textbf{Theorem \ref{ss-from-chainsat}.}
Let $\K$ be an $\LS (\K)$-tame AEC with a monster model. There exists $\chi < H_1$ such that for any $\mu \ge \chi$, if $\K$ is stable in $\mu$, there is a saturated model of cardinality $\mu$, and every limit model of cardinality $\mu$ is $\chi$-saturated, then $\K$ has no long splitting chains in $\mu$.

\textbf{Theorem \ref{strong-solvable-thm}.}
Let $\K$ be an $\LS (\K)$-tame AEC with a monster model. There exists $\chi < H_1$ such that for any $\mu \ge \chi$, if $\K$ is stable in $\mu$ and has no long splitting chains in $\mu$ then $\K$ is uniformly $(\mu', \mu')$-solvable, where $\mu' := \left(\beth_{\omega + 2} (\mu)\right)^+$.

These two theorems prove (\ref{sc0-3}) implies (\ref{sc0-1}) and (\ref{sc0-1}) implies (\ref{sc0-6}) of our main corollary, proven in detail after the proof of Corollary \ref{main-cor-unbounded}.

\begin{cor}[Main Corollary]\label{main-cor}
  Let $\K$ be a $\LS (\K)$-tame AEC with a monster model. Assume that $\K$ is stable in some cardinal greater than or equal to $\LS (\K)$. The following are equivalent:

  \begin{enumerate}
    \item\label{sc0-1} There exists $\mu_1 < H_1$ such that for every $\lambda\geq\mu_1$, $\K$ has no long splitting chains in $\lambda$.
%    \item\label{sc0-15} For unboundedly many $\kappa$, there exists a $\lambda$ such that the $U$-rank induced by $(<\kappa)$-satisfiability (see \cite[Definition 7.2]{bg-v7}) is bounded on types over $\lambda$-saturated models.
    \item\label{sc0-2} There exists $\mu_2 < H_1$ such that for every $\lambda\geq\mu_2$, there is a good $\lambda$-frame on a skeleton of $\K_\lambda$ (see Section \ref{skeleton-sec}).
    \item\label{sc0-3} There exists $\mu_3 < H_1$ such that for every $\lambda\geq\mu_3$, $\K$ has a unique limit model of cardinality $\lambda$.
    \item\label{sc0-4} There exists $\mu_4 < H_1$ such that for every $\lambda\geq\mu_4$, $\K$ has a superlimit model of cardinality $\lambda$.
    \item\label{sc0-5} There exists $\mu_5 < H_1$ such that for every $\lambda\geq\mu_5$, the union of any increasing chain of $\lambda$-saturated models is $\lambda$-saturated.
    \item\label{sc0-6} There exists $\mu_6 < H_1$ such that for every $\lambda \geq \mu_6$, $\K$ is $(\lambda, \mu_6)$-solvable.
  \end{enumerate}

  Moreover any of the above conditions also imply:

  \begin{enumerate}
    \setcounter{enumi}{6}
    \item\label{sc0-7} There exists $\mu_7 < H_1$ such that for every $\lambda \geq \mu_7$, $\K$ is stable in $\lambda$.
  \end{enumerate}
\end{cor}
\begin{remark}
  The main corollary has a global assumption of stability (in some cardinal). While stability is implied by some of the equivalent conditions (e.g.\ by (\ref{sc0-2}) or (\ref{sc0-6})) other conditions may be vacuously true if stability fails (e.g.\ (\ref{sc0-1})). Thus in order to simplify the exposition we just require stability outright.
\end{remark}
\begin{remark}
  In the context of the main corollary, if $\mu_1 \ge \LS (\K)$ is such that $\K$ is stable in $\mu_1$ and has no long splitting chains in $\mu_1$, then for any $\lambda \ge \mu_1$, $\K$ is stable in $\lambda$ and has no long splitting chains in $\lambda$ (see Fact \ref{ss-stable}). In other words, superstability defined in terms of no long splitting chains transfers up.
\end{remark}
\begin{remark}
  In (\ref{sc0-3}), one can also require the following strong version of uniqueness of limit models:  if $M_0, M_1, M_2 \in \K_{\lambda}$ and both $M_1$ and $M_2$ are limit over $M_0$, then $M_1 \cong_{M_0} M_2$ (i.e.\ the isomorphism fixes the base). This is implied by (\ref{sc0-2}): see Fact \ref{good-frame-uq-limit}.
\end{remark}
\begin{remark}\label{stab-ss-rmk}
  At the time this paper was first circulated (July 2015), we did not know whether (\ref{sc0-7}) implied the other conditions. This has recently been proven by the second author \cite{stab-spec-v4}.
\end{remark}

Note that in Corollary \ref{main-cor}, we can let $\mu$ be the maximum of the $\mu_\ell$'s and then each property will hold above $\mu$. Interestingly however, the proof of Corollary \ref{main-cor} does \emph{not} tell us that the \emph{least} cardinals $\mu_\ell$ where the corresponding properties holds are all equal. In fact, it uses tameness heavily to move from one cardinal to the next and uses e.g.\ that one equivalent definition holds below $\lambda$ to prove that another definition holds at $\lambda$. Showing equivalence of these definitions cardinal by cardinals, or even just showing that the least cardinals where the properties hold are all equal seems much harder. We also show that we can ask only for each property to hold in \emph{a single high-enough} cardinals below $H_1$ (but again the least such cardinal may not be the same for each property, see Corollary \ref{main-cor-unbounded}). In general, we suspect that the problem of computing the minimal value of the cardinals $\mu_\ell$ could play a role similar to the computation of the first stability cardinal for a first-order theory (which led to the development of forking, see e.g.\ the introduction of \cite{primer}).

We discuss earlier work. Shelah \cite[Chapter II]{shelahaecbook} introduced good $\lambda$-frames (a local axiomatization of first-order forking in a superstable theory, see more in Section \ref{good-frame-subsec}) and attempts to develop a theory of superstability in this context. He proves for example the uniqueness of limit models (see Fact \ref{good-frame-uq-limit}, so (\ref{sc0-2}) implies (\ref{sc0-3}) in the main theorem is due to Shelah) and (with strong assumptions, see below) the fact that the union of a chain (of length strictly less than $\lambda^{++}$) of saturated models of cardinality $\lambda^+$ is saturated \cite[II.8]{shelahaecbook}. From this he deduces the existence of a good $\lambda^+$-frame on the class of $\lambda^+$-saturated models of $\K$ and goes on to develop a theory of prime models, regular types, independent sequences, etc.\ in \cite[Chapter III]{shelahaecbook}. The main issue with Shelah's work is that it does not make any global model-theoretic hypotheses (such as tameness or even just amalgamation) and hence often relies on set-theoretic assumptions as well as strong local model-theoretic hypotheses (few models in several cardinals). For example, Shelah's construction of a good frame in the local setup \cite[II.3.7]{shelahaecbook} uses categoricity in two successive cardinals, few models in the next, as well as several diamond-like principles.

By making more global hypotheses, building a good frame becomes easier and can be done in ZFC (see \cite{ss-tame-jsl} or \cite[Chapter IV]{shelahaecbook}). Recently, assuming a monster model and tameness (a locality property of types introduced by VanDieren and the first author, see Definition \ref{shortness-def}), progress have been made in the study of superstability defined in terms of no long splitting chains. Specifically, \cite[5.6]{ss-tame-jsl} proved (\ref{sc0-1}) implies (\ref{sc0-7}). Partial progress in showing (\ref{sc0-1}) implies (\ref{sc0-2}) is made in \cite{ss-tame-jsl} and \cite{indep-aec-apal} but the missing piece of the puzzle, that (\ref{sc0-1}) implies (\ref{sc0-5}), is proven in \cite{bv-sat-v3}. From these results, it can be deduced that \emph{(\ref{sc0-1}) implies (\ref{sc0-2})-(\ref{sc0-5})} (see \cite[7.1]{bv-sat-v3}). Shelah has shown that (\ref{sc0-2}) implies (\ref{sc0-3}), see Fact \ref{good-frame-uq-limit}. Some implications between variants of (\ref{sc0-3}), (\ref{sc0-4}) and (\ref{sc0-5}) are also straightforward (see Fact \ref{local-implications}), though one has to be careful about where the class is stable (the existence of a limit model of cardinality $\lambda$ implies stability in $\lambda$, but not the fact that the union of a chain of $\lambda$-saturated models is $\lambda$-saturated). In the proof of Corollary \ref{main-cor-unbounded}, we end up using a single technical condition, ($\ref{sc0-3}^\ast$), asserting that limit models have a certain degree of saturation. It is quite easy to see that (\ref{sc0-3}), (\ref{sc0-4}), and (\ref{sc0-5}) all imply ($\ref{sc0-3}^\ast$). Finally, (\ref{sc0-6}) directly implies (\ref{sc0-4}) from its definition (see Section \ref{solvability-subsec}).

Thus as noted before the main contributions of this paper are (\ref{sc0-3}) (or really ($\ref{sc0-3}^\ast$)) implies (\ref{sc0-1}) and (\ref{sc0-1}) implies (\ref{sc0-6}). In Theorem \ref{solv-transfer} it is shown that, assuming a monster model and tameness, solvability in \emph{some} high-enough cardinal implies solvability in \emph{all} high-enough cardinals. Note that Shelah asks (inspired by the analogous question for categoricity) in \cite[Question N.4.4]{shelahaecbook} what the solvability spectrum can be (in an arbitrary AEC). Theorem \ref{solv-transfer} provides a partial answer under the additional assumptions of a monster model and tameness. The proof notices that a powerful results of Shelah and Villaveces \cite{shvi635} (deriving no long splitting chains from categoricity) can be adapted to our setup (see Fact \ref{ns-lc-fact} and Corollary \ref{ns-lc-cor}). Shelah also asks \cite[Question N.4.5]{shelahaecbook} about the superlimit spectrum. In our context, we can show that if there is a high-enough \emph{stability} cardinal $\lambda$ with a superlimit model, then $\K$ has a superlimit on a tail of cardinals (see Corollary \ref{main-cor-unbounded}). We do not know if the hypothesis that $\lambda$ is a stability cardinal is needed (see Question \ref{superlimit-question}).

Since this paper was first circulated (July 2015), several related results have been proven. VanDieren \cite{vandieren-symmetry-apal, vandieren-chainsat-apal} gives some relationships between versions of (\ref{sc0-3}) and (\ref{sc0-5}) in a single cardinal (with (\ref{sc0-1}) as a background assumption). This is done without assuming tameness, using very different technologies than in this paper. This work is applied to the tame context in \cite{vv-symmetry-transfer-v3}, showing for example that (\ref{sc0-1}) implies (\ref{sc0-3}) holds cardinal by cardinal. A recent preprint of the second author \cite{stab-spec-v4} studies the model theory of strictly stable tame AECs, establishing in particular that stability on a tail implies no long splitting chains (see Remark \ref{stab-ss-rmk}).

We do not know how to prove analogs to the last two properties of Fact \ref{fo-superstab}. Note also that, while the analogous result is known for stability (see Fact \ref{stab-spectrum}), we do not know whether no long splitting chains should hold below the Hanf number:

\begin{question}
  Let $\K$ be a $\LS (\K)$-tame AEC with a monster model. Assume that there exists $\mu \ge \LS (\K)$ such that $\K$ is stable in $\mu$ and has no long splitting chains in $\mu$. Is the least such $\mu$ below $H_1$?
\end{question}

The background required to read this paper is a solid knowledge of AECs (for example Chapters 4-12 of Baldwin's book \cite{baldwinbook09} or the upcoming \cite{grossbergbook}). We rely on the first ten sections of \cite{indep-aec-apal}, as well as on the material in \cite{sv-infinitary-stability-afml, bv-sat-v3}.

At the beginning of Sections \ref{forking-sec} and \ref{solvability-sec}, we make \emph{global} hypotheses that hold until the end of the section (unless said otherwise). This is to make the statements of several technical lemmas more readable. We will repeat the global hypotheses in the statement of major theorems.

This paper was written while the second author was working on a Ph.D.\ thesis under the direction of the first author at Carnegie Mellon University. He would like to thank Professor Grossberg for his guidance and assistance in his research in general and in this work specifically. We also thank Will Boney and a referee for feedback that helped improve the presentation of the paper.

\section{Preliminaries}

We assume familiarity with a basic text on AECs such as \cite{baldwinbook09} or \cite{grossbergbook} and refer the reader to the preliminaries of \cite{sv-infinitary-stability-afml} for more details and motivations on the concepts used in this paper.

We use $\K$ (boldface) to denote a class of models together with an ordering (written $\lea$). We will often abuse notation and write for example $M \in \K$. When it becomes necessary to consider only a class of models without an ordering, we will write $K$ (no boldface).

Throughout all this paper, $\K$ is a fixed AEC. Most of the time, $\K$ will have amalgamation, joint embedding, and arbitrarily large models. In this case we say that \emph{$\K$ has a monster model}.

The notion of tameness was introduced by Grossberg and VanDieren \cite{tamenessone} as a useful assumption to prove upward stability results. In \cite{tamenesstwo, tamenessthree}, several cases of Shelah's eventual categoricity conjecture were established in tame AECs. Boney \cite{tamelc-jsl} derived from the existence of a class of strongly compact cardinals that all AECs are tame. In a forthcoming paper Boney and Unger \cite{lc-tame-v3-toappear} establish that if every AEC is tame then a proper class of large cardinals exists.  Thus tameness (for all AECs) is a large cardinal axioms. We believe that this is evidence for the assertion that tameness is a new interesting model-theoretic property, a new dichotomy\footnote{Consider, for example, the statement that in a monster model for a first-order theory $T$, for every sufficiently long sequence $\BI$ there exists a subsequence $\BJ \subseteq \BI$ such that $\BJ$ is indiscernible. In general, this is a large cardinal axiom, but it is known to be true when $T$ is on the good side of a dividing line (in this case stability). We believe that the situation for tameness is similar.}, that should follow (see \cite[Conjecture 1.5]{tamenessthree}) from categoricity in a ``high-enough'' cardinal.

\begin{defin}[3.2 in \cite{tamenessone}]\label{shortness-def}
  Let $\chi \ge \LS (\K)$ be a cardinal. $\K$ is \emph{$\chi$-tame} if for any $M \in \K_{\ge \chi}$ and any $p \neq q$ in $\gS (M)$, there exists $M_0 \in \K_{\chi}$ such that $p \rest M_0 \neq q \rest M_0$. We similarly define $(<\chi)$-tame (when $\chi > \LS (\K)$).
  
  We say that \emph{$\K$ is tame} provided there exists a cardinal $\chi$ such that\footnote{As opposed to \cite[3.3]{tamenessone}, we do \emph{not} require that $\chi < H_1$.} $\K$ is $\chi$-tame.
\end{defin}
\begin{remark}
  If $\K$ is $\chi$-tame for $\chi > \LS (\K)$, the class $\K' := \K_{\ge \chi}$ will be an $\LS (\K')$-tame AEC. Hence we will usually directly assume that $\K$ is $\LS (\K)$-tame.
\end{remark}

We will use the equivalence between stability and the order property under tameness \cite[4.13]{sv-infinitary-stability-afml}:

\begin{fact}\label{stab-spectrum}
  Assume that $\K$ is $\LS (\K)$-tame and has a monster model. The following are equivalent:

  \begin{enumerate}
  \item $\K$ is stable in some cardinal greater than or equal to $\LS (\K)$.
  \item There exists $\mu < H_1$ such that $\K$ is stable in all $\lambda \ge \LS (\K)$ such that $\lambda = \lambda^{\mu}$.
  \item $\K$ does not have the $\LS (\K)$-order property.
  \end{enumerate}
\end{fact}

\subsection{Superstability and no long splitting chains}

A definition of superstability analogous to $\kappa (T) = \aleph_0$ in first-order model theory has been studied in AECs (see \cite{shvi635, gvv-mlq, vandierennomax, nomaxerrata, ss-tame-jsl, indep-aec-apal}). Since it is not immediately obvious what forking should be in that framework, the more rudimentary independence relation of $\lambda$-splitting is used in the definition. Since in AECs, types over models are much better behaved than types over sets, it does not make sense in general to ask for every type to not split over a finite set\footnote{But see \cite[C.14]{ap-universal-v10} where a notion of forking over set is constructed from categoricity in a universal class.}. Thus we require that every type over the union of a chain does not split over a model in the chain. For technical reasons (it is possible to prove that the condition follows from categoricity), we require the chain to be increasing with respect to universal extension. Definition \ref{loc-card-def} rephrases (\ref{sc0-1}) in Corollary \ref{main-cor}:

\begin{defin}\label{loc-card-def} \
Let $\lambda \ge \LS (K)$. We say $\K$ has \emph{no long splitting chains in $\lambda$} if for any limit $\delta < \lambda^+$, any increasing $\seq{M_i : i < \delta}$ in $\K_\lambda$ with $M_{i + 1}$ universal over $M_i$ for all $i < \delta$, any $p \in \gS (\bigcup_{i < \delta} M_i)$, there exists $i < \delta$ such that $p$ does not $\lambda$-split over $M_i$.
\end{defin}

\begin{remark}
  The condition in Definition \ref{loc-card-def} first appears in \cite[Question 6.1]{sh394}. In \cite[15.1]{baldwinbook09}, it is written as\footnote{Of course, the $\kappa$ notation has a long history, appearing first in \cite{sh3}.} $\kappa (\K, \lambda) = \aleph_0$. We do not adopt this notation, since it blurs out the distinction between forking and splitting, and does not mention that only a certain type of chains are considered. A similar notation is in \cite[3.16]{indep-aec-apal}: $\K$ has no long splitting chains in $\lambda$ if and only if $\clc{1} (\insp{\lambda} (\K_\lambda), <_{\text{univ}}) = \aleph_0$.
\end{remark}

In tame AECs with a monster model, no long splitting chains transfers upward:

\begin{fact}[10.10 in \cite{indep-aec-apal}]\label{ss-stable}
  Let $\K$ be an AEC with a monster model and let $\LS (\K) \le \lambda \le \mu$. If $\K$ is stable in $\lambda$ and has no long splitting chains in $\lambda$, then $\K$ is stable in $\mu$ and has no long splitting chains in $\mu$.
\end{fact}

\subsection{Definitions of saturated}\label{sat-def-subsec}

The search for a good definition of ``saturated'' in AECs is central. We quickly review various possible notions and cite some basic facts about them, including basic implications.

Implicit in the definition of no long splitting chains is the notion of a \emph{limit model}. It plays a central role in the study of AECs that do not necessarily have amalgamation \cite{shvi635} (their study in this context was continued in \cite{vandierennomax, nomaxerrata}). We use the notation and basic definitions from \cite{gvv-mlq}. The two basic facts about limit models (in an AEC with a monster model) are:

\begin{enumerate}
\item Existence: If $\K$ is stable in $\lambda$, then for every $M$ and every limit $\delta < \lambda^+$ there is a $(\lambda, \delta)$-limit over $M$.
\item Uniqueness: Any two limit models of the same length are isomorphic.
\end{enumerate}

Uniqueness of limit models that are \emph{not} of the same cofinality is a key concept which is equivalent to superstability in first-order model theory.

A second notion of saturation can be defined using Galois types (when $\K$ has a monster model): for $\lambda > \LS (\K)$, say $M$ is \emph{$\lambda$-saturated} if every type over a $\lea$-substructure of $M$ of size less than $\lambda$ is realized inside $M$. We will write $\Ksatp{\lambda}$ for the class of $\lambda$-saturated models in $\K$. 

A third notion of saturation appears in \cite[3.1(1)]{sh88}\footnote{We use the definition in \cite[N.2.4(4)]{shelahaecbook} which requires in addition that the model be universal.}. The idea is to encode a generalization of the fact that a union of saturated models should be saturated.

\begin{defin}\label{sl-def}
  Let $M \in \K$ and let $\lambda \ge \LS (\K)$. $M$ is called \emph{superlimit in $\lambda$ if:}

  \begin{enumerate}
    \item $M \in \K_\lambda$.
    \item $M$ is ``properly universal'': For any $N \in \K_\lambda$, there exists $f: N \rightarrow M$ such that $f[N] \lta M$.
    \item Whenever $\seq{M_i : i < \delta}$ is an increasing chain in $\K_\lambda$, $\delta < \lambda^+$ and $M_i \cong M$ for all $i < \delta$, then $\bigcup_{i < \delta} M_i \cong M$.
  \end{enumerate}
\end{defin}

The following local implications between the three definitions are known:

\begin{fact}[Local implications]\label{local-implications}
  Assume that $\K$ has a monster model. Let $\lambda \ge \LS (\K)$ be such that $\K$ is stable in $\lambda$.

  \begin{enumerate}
  \item\label{local-1} If $\chi \in [\LS (\K)^+, \lambda]$ is regular, then any $(\lambda, \chi)$-limit model is $\chi$-saturated.
  \item\label{local-2} If $\lambda > \LS (\K)$ and $\lambda$ is regular, then $M \in \K_{\lambda}$ is saturated if and only if $M$ is $(\lambda, \lambda)$-limit.
  \item\label{local-3} If $\lambda > \LS (\K)$, then any two limit models of size $\lambda$ are isomorphic if and only if every limit model of size $\lambda$ is saturated.
  \item\label{local-4} If $M \in \K_{\lambda}$ is superlimit, then for any limit $\delta < \lambda^+$, $M$ is $(\lambda, \delta)$-limit and (if $\lambda > \LS (\K)$) saturated.
  \item\label{local-5} Assume that $\lambda > \LS (\K)$ and there exists a saturated model $M$ of size $\lambda$. Then $M$ is superlimit if and only if in $\K_\lambda$, the union of any increasing chain (of length strictly less than $\lambda^+$) of saturated models is saturated.
  \end{enumerate}
\end{fact}
\begin{proof}
  (\ref{local-1}), (\ref{local-2}), and (\ref{local-3}) are straightforward from the basic facts about limit models and the uniqueness of saturated models. (\ref{local-4}) is by \cite[2.3.10]{drueckthesis} and the previous parts. (\ref{local-5}) then follows.
\end{proof}
\begin{remark}
  (\ref{local-3}) is stated for $\lambda$ regular in \cite[2.3.12]{drueckthesis} but the argument above shows that it holds for any $\lambda$.
\end{remark}

\subsection{Skeletons}\label{skeleton-sec}

The notion of a skeleton was introduced in \cite[Section 5]{indep-aec-apal} and is meant to be an axiomatization of a subclass of saturated models of an AEC. It is mentioned in (\ref{sc0-2}) of the main corollary.

Recall the definition of an abstract class, due to the first author \cite{grossbergbook} (or see \cite[2.7]{sv-infinitary-stability-afml}): it is a pair $\K' = (K', \leap{\K'})$ so that $K'$ is a class of $\tau$-structures in a fixed vocabulary $\tau = \tau (\K')$, closed under isomorphisms, and $\leap{\K'}$ is a partial order on $K'$ which respects isomorphisms and extends the $\tau$-substructure relation. 

\begin{defin}[5.3 in \cite{indep-aec-apal}]\label{skel-def}
  A \emph{skeleton} of an abstract class $\K^\ast$ is an abstract class $\K'$ such that:

  \begin{enumerate}
  \item $K' \subseteq K^\ast$ and for $M, N \in \K'$, $M \leap{\K'} N$ implies $M \leap{\K^\ast} N$.
  \item $\K'$ is dense in $\K^\ast$: For any $M \in \K^\ast$, there exists $M' \in \K'$ such that $M \leap{\K^\ast} M'$.
  \item If $\alpha$ is a (not necessarily limit) ordinal and $\seq{M_i : i < \alpha}$ is a strictly $\leap{\K^\ast}$-increasing chain in $\K'$, then there exists $N \in \K'$ such that $M_i \leap{\K'} N$ and\footnote{Note that if $\alpha$ is limit this follows.} $M_i \neq N$ for all $i < \alpha$.
  \end{enumerate}
\end{defin}

\begin{example}
  Let $\lambda \ge \LS (\K)$. Assume that $\K$ is stable in $\lambda$, has amalgamation and no maximal models in $\lambda$. Let $K'$ be the class of limit models of size $\lambda$ in $\K$. Then $(K', \lea)$ (or even $K'$ ordered with ``being equal or universal over'') is a skeleton of $\K_\lambda$.
\end{example}
\begin{remark}
  If $\K'$ is a skeleton of $\K_\lambda$ and $\K'$ itself generates an AEC, then $M \leap{\K'} N$ if and only if $M, N \in \K'$ and $M \lea N$. This is because of the third clause in the definition of a skeleton (used with $\alpha = 2$) and the coherence axiom.
\end{remark}

We can define notions such as amalgamation and Galois types for any abstract class (see the preliminaries of \cite{sv-infinitary-stability-afml}). The properties of a skeleton often correspond to properties of the original AEC:

\begin{fact}\label{skeleton-facts}
  Let $\lambda \ge \LS (\K)$ and assume that $\K$ has amalgamation in $\lambda$. Let $\K'$ be a skeleton of $\K_\lambda$.

  \begin{enumerate}
  \item\label{skel-1} For $P$ standing for having no maximal models in $\lambda$, being stable in $\lambda$, or having joint embedding in $\lambda$, $\K$ has $P$ if and only if $\K'$ has $P$.
  \item\label{skel-2} Assume that $\K$ has joint embedding in $\lambda$ and for every limit $\delta < \lambda^+$ and every $N \in  \K'$ there exists $N' \in \K'$ which is $(\lambda, \delta)$-limit over $N$ (in the sense of $\K'$).
    \begin{enumerate}
    \item\label{skel-2a} Let $M, M_0 \in \K'$ and let $\delta < \lambda^+$ be a limit ordinal. Then $M$ is $(\lambda, \delta)$-limit over $M_0$ in the sense of $\K'$ if and only $M$ is $(\lambda, \delta)$-limit over $M_0$ in the sense of $\K$.
    \item\label{skel-2b} $\K'$ has no long splitting chains in $\lambda$ if and only if $\K$ has no long splitting chains in $\lambda$.
    \end{enumerate}
  \end{enumerate}
\end{fact}
\begin{proof}
  (\ref{skel-1}) is by \cite[5.8]{indep-aec-apal}. As for (\ref{skel-2a}), (\ref{skel-2b}), note first that the hypotheses of (\ref{skel-2}) imply (by (\ref{skel-1})) that $\K$ is stable in $\lambda$ and has no maximal models in $\lambda$. In particular, limit models of size $\lambda$ exist in $\K$.

  Let us prove (\ref{skel-2a}). If $M$ is $(\lambda, \delta)$-limit over $M_0$ in the sense of $\K'$, then it is straightforward to check that the chain witnessing it will also witness that $M$ is $(\lambda, \delta)$-limit over $M_0$ in the sense of $\K$. For the converse, observe that by assumption there exists a $(\lambda, \delta)$-limit $M'$ over $M_0$ in the sense of $\K'$. Furthermore, by what has just been observed $M'$ is also limit in the sense of $\K$, hence by uniqueness of limit models of the same length, $M' \cong_{M_0} M$. Therefore $M$ is also $(\lambda, \delta)$-limit over $M_0$ in the sense of $\K'$. The proof of (\ref{skel-2b}) is similar, see \cite[6.7]{indep-aec-apal}.
\end{proof}

\subsection{Good frames}\label{good-frame-subsec}

Good frames are a local axiomatization of forking in a first-order superstable theories. They are introduced in \cite[Chapter II]{shelahaecbook}. We will use the definition from \cite[8.1]{indep-aec-apal} which is weaker and more general than Shelah's, as it does not require the existence of a superlimit (as in \cite{jrsh875}). As opposed to \cite{indep-aec-apal}, we allow good frames that are \emph{not} type-full: we only require the existence of a set of well-behaved basic types satisfying some density property (see \cite[Chapter II]{shelahaecbook} for more). Note however that Remark \ref{type-full-rmk} says that in the context of the main theorem the existence of a good frame implies the existence of a \emph{type-full} good frame (possibly over a different class).

In \cite[8.1]{indep-aec-apal}, the underlying class of the good frame consists only of models of size $\lambda$. Thus when we say that there is a good $\lambda$-frame on a class $\K'$, we mean the underlying class of the good frame is $\K'$, and the axioms of good frames will require that $\K'$ generates a non-empty AEC with amalgamation in $\lambda$, joint embedding in $\lambda$, no maximal models in $\lambda$, and stability in $\lambda$.

The only facts that we will use about good frames are:

\begin{fact}\label{good-frame-uq-limit}
  Let $\lambda \ge \LS (\K)$. If there is a good $\lambda$-frame on a skeleton of $\K_{\lambda}$, then $\K$ has a unique limit model of size $\lambda$. Moreover, for any $M_0, M_1, M_2 \in \K_{\lambda}$, if both $M_1$ and $M_2$ are limit over $M_0$, then $M_1 \cong_{M_0} M_2$ (i.e.\ the isomorphism fixes $M_0$).
\end{fact}
\begin{proof}
  Let $\K'$ be the skeleton of $\K_\lambda$ which is the underlying class of the good $\lambda$-frame. By \cite[II.4.8]{shelahaecbook} (see \cite[9.2]{ext-frame-jml} for a detailed proof), $\K'$ has a unique limit model of size $\lambda$ (and the moreover part holds for $\K'$). By Fact \ref{skeleton-facts}(\ref{skel-2a}), this must also be the unique limit model of size $\lambda$ in $\K$ (and the moreover part holds in $\K$ too).
\end{proof}

\begin{fact}\label{good-frame-existence}
  Assume that $\K$ has a monster model and is $\LS (\K)$-tame. If $\mu < H_1$ is such that $\K$ is stable in $\mu$ and has no long splitting chains in $\mu$, then there exists $\lambda_0 < H_1$ such that for all $\lambda \ge \lambda_0$, there is a good $\lambda$-frame on $\Ksatp{\lambda}_\lambda$. Moreover, $\Ksatp{\lambda}_\lambda$ is a skeleton of $\K_\lambda$, $\K$ is stable in $\lambda$, any $M \in \Ksatp{\lambda}_\lambda$ is superlimit, and the union of any increasing chain of $\lambda$-saturated models is $\lambda$-saturated. 
\end{fact}
\begin{proof}
  First assume that $\K$ has no long splitting chains in $\LS (\K)$ and is stable in $\LS (\K)$. By \cite[7.1]{bv-sat-v3}, there exists $\lambda_0 < \beth_{\left(2^{\mu^+}\right)^+}$ such that for any $\lambda \ge \lambda_0$, any increasing chain of $\lambda$-saturated models is $\lambda$-saturated and there is a good $\lambda$-frame on $\Ksatp{\lambda}_\lambda$. That any $M \in \Ksatp{\lambda}_\lambda$ is a superlimit (Fact \ref{local-implications}(\ref{local-5})) and $\Ksatp{\lambda}_\lambda$ is a skeleton of $\K_\lambda$ easily follows, and stability in $\lambda$ is given (for example) by Fact \ref{skeleton-facts}(\ref{skel-1}).

  Now by \cite[6.12]{bv-sat-v3}, we more precisely have that if $\K$ has no long splitting chains in $\mu$ and is stable in $\mu$ (with $\mu \ge \LS (\K)$) and $(<\LS (\K))$-tame (tameness being defined using types over sets), then the same conclusion holds with $\beth_{\left(2^{\mu^+}\right)^+}$ replaced by $H_1$. Now the use of $(<\LS (\K))$-tameness is to derive that there exists $\chi < H_1$ so that $\K$ does not have a certain order property of length $\chi$, but \cite{bv-sat-v3} relies on an older version of \cite{sv-infinitary-stability-afml} which proves Fact \ref{stab-spectrum} assuming $(<\LS (\K))$-tameness instead of $\LS (\K)$-tameness. In the current version of \cite{sv-infinitary-stability-afml}, it is shown that $\LS (\K)$-tameness suffices, thus the arguments of \cite{bv-sat-v3} go through assuming $\LS (\K)$-tameness instead of $(<\LS (\K))$-tameness.
\end{proof}

\subsection{Solvability}\label{solvability-subsec}

Solvability appears as a possible definition of superstability for AECs in \cite[Chapter IV]{shelahaecbook}. The definition uses Ehrenfeucht-Mostowski models and we assume the reader has some familiarity with them, see for example \cite[Section 6.2]{baldwinbook09} or \cite[IV.0.8]{shelahaecbook}.

\begin{defin}\label{em-def} \
  \begin{enumerate}
    \item A countable set $\Phi = \{p_n : n < \omega\}$ is \emph{proper for linear orders} if the $p_n$'s are an increasing sequence of $n$-variable quantifier-free types in a fixed vocabulary $\tau (\Phi)$ which are satisfied by a sequence of indiscernibles. As usual, such a set $\Phi$ determines an EM-functor from linear orders into $\tau (\Phi)$-structures, mapping a linear order $I$ to $\EM (I, \Phi)$ and taking suborders to substructures.
    \item \cite[IV.0.8]{shelahaecbook} For $\mu \ge \LS (\K)$, let $\Upsilon_{\mu}[\K]$ be the set of $\Phi$ proper for linear orders with $\tau (\K) \subseteq \tau (\Phi)$, $|\tau (\Phi)| \le \mu$, and such that the $\tau (\K)$-reduct $\EM_{\tau (\K)} (I, \Phi)$ is a functor from linear orders into members of $\K$ of cardinality at most $|I| + \mu$. Such a $\Phi$ is called an \emph{EM blueprint for $\K$}.
  \end{enumerate}
\end{defin}
\begin{defin}\label{solvability-def}
  Let $\LS (\K) \le \mu \le \lambda$.
  
  \begin{enumerate}
    \item \cite[IV.1.4(1)]{shelahaecbook} We say that \emph{$\Phi$ witnesses $(\lambda, \mu)$-solvability} if:
  \begin{enumerate}
  \item $\Phi \in \Upsilon_{\mu}[\K]$.
  \item If $I$ is a linear order of size $\lambda$, then $\EM_{\tau (\K)} (I, \Phi)$ is superlimit in $\lambda$ for $\K$, see Definition \ref{sl-def}.
  \end{enumerate}
  
  $\K$ is \emph{$(\lambda, \mu)$-solvable} if there exists $\Phi$ witnessing $(\lambda, \mu)$-solvability.
\item $\K$ is \emph{uniformly $(\lambda, \mu)$-solvable} if there exists $\Phi$ such that for all $\lambda' \ge \lambda$, $\Phi$ witnesses $(\lambda', \mu)$-solvability.
  \end{enumerate}
\end{defin}
\begin{fact}[IV.0.9 in \cite{shelahaecbook}]\label{EM-existence}
  Let $\K$ be an AEC and let $\mu \ge \LS (\K)$. Then $\K$ has arbitrarily large models if and only if $\Upsilon_\mu[\K] \neq \emptyset$.
\end{fact}

We give some more manageable definitions of solvability ((\ref{equiv-cond-3}) is the one we will use). Shelah already mentions one of them on \cite[p.~61]{shelahaecbook} (but does not prove it is equivalent).

\begin{lem}\label{solvability-equiv}
  Let $\LS (\K) \le \mu \le \lambda$. The following are equivalent.

  \begin{enumerate}
    \item\label{equiv-cond-1} $\K$ is [uniformly] $(\lambda, \mu)$-solvable.
    \item\label{equiv-cond-2} There exists $\tau' \supseteq \tau (\K)$ with $|\tau'| \le \mu$ and $\psi \in \Ll_{\mu^+, \omega} (\tau')$ such that:

      \begin{enumerate}
        \item $\psi$ has arbitrarily large models.
        \item {[}For all $\lambda' \ge \lambda$], if $M \models \psi$ and $\|M\| = \lambda$ [$\|M\| = \lambda'$], then $M \rest \tau (\K)$ is in $\K$ and is superlimit.
      \end{enumerate}
    \item\label{equiv-cond-3} There exists $\tau' \supseteq \tau (\K)$ and an AEC $\K'$ with $\tau(\K') = \tau'$, $\LS (\K') \le \mu$ such that:
      \begin{enumerate}
        \item $\K'$ has arbitrarily large models.
        \item {[}For all $\lambda' \ge \lambda$], if $M \in \K'$ and $\|M\| = \lambda$ [$\|M\| = \lambda'$], then $M \rest \tau (\K)$ is in $\K$ and is superlimit.
      \end{enumerate}
  \end{enumerate}
\end{lem}
\begin{proof} \
  \begin{itemize}
    \item \underline{(\ref{equiv-cond-1}) implies (\ref{equiv-cond-2})}:
      Let $\Phi$ witness $(\lambda, \mu)$-solvability and write $\Phi = \{p_n \mid n < \omega\}$. Let $\tau' := \tau (\Phi) \cup \{P, <\}$, where $P$, $<$ are symbols for a unary predicate and a binary relation respectively. Let $\psi \in \Ll_{\mu^+, \omega} (\tau')$ say:
      \begin{enumerate}
        \item $(P, <)$ is a linear order.
        \item For all $n < \omega$ and all $x_0 < \cdots < x_{n - 1}$ in $P$, $x_0 \ldots x_{n - 1}$ realizes $p_n$.
        \item For all $y$, there exists $n < \omega$, $x_0 < \cdots < x_{n - 1}$ in $P$, and $\rho$ an $n$-ary term of $\tau (\Phi)$ such that $y = \rho (x_0, \ldots, x_{n - 1})$.
      \end{enumerate}

       Then if $M \models \psi$, $M \rest \tau = \EM_{\tau (\K)} (P^M, \Phi)$ (and by solvability if $\|M\| = \lambda$ then $M$ is superlimit in $\K$). Conversely, if $M = \EM_{\tau (\K)} (I, \Phi)$, we can expand $M$ to a $\tau'$-structure $M'$ by letting $(P^{M'}, <^{M'}) := (I, <)$. Thus $\psi$ is as desired.
    \item \underline{(\ref{equiv-cond-2}) implies (\ref{equiv-cond-3})}: Given $\tau'$ and $\psi$ as given by (\ref{equiv-cond-2}), Let $\Psi$  be a fragment of $\Ll_{\mu^+, \omega} (\tau')$ containing $\psi$ of size $\mu$ and let $\K'$ be $\Mod (\psi)$ ordered by $\lee_{\Psi}$. Then $\K'$ is as desired for (\ref{equiv-cond-3}).
    \item \underline{(\ref{equiv-cond-3}) implies (\ref{equiv-cond-1})}: Directly from Fact \ref{EM-existence}.
  \end{itemize}
\end{proof}

\section{Forking and averages in stable AECs}\label{forking-sec}

In the introduction to his book \cite[p.~61]{shelahaecbook}, Shelah asserts (without proof) that in the first-order context solvability (see Section \ref{solvability-subsec}) is equivalent to superstability. We aim to give a proof (see Corollary \ref{solvability-fo}) and actually show (assuming amalgamation, stability, and tameness) that solvability is equivalent to any of the definitions in the main theorem. First of all, if there exists $\mu$ such that $\K$ is $(\lambda, \mu)$-solvable for all high-enough $\lambda$, then in particular $\K$ has a superlimit in all high-enough $\lambda$, so we obtain (\ref{sc0-4}) in the main corollary. We work toward a converse. The proof is similar to that in \cite{fcp-saturation}: we aim to code saturated models using their characterization with average of sequences (the crucial result for this is Lemma \ref{saturation-charact}). In this section, we  use the theory of averages in AECs (as developed by Shelah in \cite[Chapter V.A]{shelahaecbook2} and by Boney and the second author in \cite{bv-sat-v3}) to give a new characterization of forking (Lemma \ref{forking-charact}). We also prove the key result for (\ref{sc0-5}) implies (\ref{sc0-1}) in the main corollary (Theorem \ref{ss-from-chainsat}). All throughout, we assume:

\begin{hypothesis}\label{global-nf-hyp} \
  \begin{enumerate}
  \item $\K$ has a monster model $\sea$ (we work inside it).
  \item $\K$ is $\LS (\K)$-tame.
  \item $\K$ is stable in some cardinal greater than or equal to $\LS (\K)$.
  \end{enumerate}
\end{hypothesis}

  We set $\kappa := \LS (\K)^+$ and work in the setup of \cite[Section 5]{bv-sat-v3}. In particular we think of Galois types of size $\LS (\K)$ as formulas and think of bigger Galois types as set of such formulas. That is, we work inside the Galois Morleyization of $\K$ (see \cite[3.3, 3.16]{sv-infinitary-stability-afml}). We encourage the reader to have a copy of both \cite{sv-infinitary-stability-afml} and \cite{bv-sat-v3} open, since we will cite from there freely and use basic notation and terminology ($\chi$-convergent, $\chi$-based, $(\chi_0, \chi_1, \chi_2)$-Morley, $\Av_\chi (\BI / A)$ etc.) often without even an explicit citation. We will say that $p \in \gS^{<\kappa} (M)$ \emph{does not syntactically split} over $M_0 \lea M$ if it does not split in the syntactic sense of \cite[5.7]{bv-sat-v3} (that is, it does not split in the usual first-order sense when we think of Galois types of size $\LS (\K)$ as formulas). Note that several results from \cite{bv-sat-v3} that we quote assume $(<\LS (\K))$-tameness (defined in terms of Galois types over sets). However, as argued in the proof of Fact \ref{good-frame-existence}, $\LS (\K)$-tameness suffices.

  We will define several other cardinals $\chi_0 < \chi_0' < \chi_1 < \chi_1' < \chi_2$ (see Notation \ref{notation-1}, \ref{notation-2}, and \ref{notation-3}). The reader can simply see them as ``high-enough'' cardinals with reasonable closure properties. If $\chi_0$ is chosen reasonably, we will have $\chi_2 < H_1$.

The letters $\BI$, $\BJ$ will denote sequences of tuples of length strictly less than $\kappa$. We will use the same conventions as in \cite[Section 5]{bv-sat-v3}. Note that the sequences \emph{may be indexed by arbitrary linear orders}. 

By Facts \ref{stab-spectrum} and \cite[I.4.5(3)]{sh394} (recalling that there is a global assumption of stability in this section), we have:

\begin{fact}\label{stab-op}
  There exists $\chi_0 < H_1$ such that $\K$ does not have the $\LS (\K)$-order property of length $\chi_0$.
\end{fact}

Another property of $\chi_0$ is the following more precise version of Fact \ref{stab-spectrum} (see \cite{sv-infinitary-stability-afml} on how to translate Shelah's syntactic version to AECs):

\begin{fact}[Theorem V.A.1.19 in \cite{shelahaecbook2}]\label{stab-spectrum-2}
  If $\lambda = \lambda^{\chi_0}$, then $\K$ is stable in $\lambda$. In particular, $\K$ is stable in $\chi_0'$.
\end{fact}

The following notation will be convenient:

\begin{notation}\label{notation-1}
  Let $\chi_0$ be any regular cardinal such that $\chi_0 \ge 2^{\LS (\K)}$ and $\K$ does not have the $\LS (\K)$-order property of length $\chi_0^+$. For a cardinal $\lambda$, let $\gamma (\lambda) := (2^{2^{\lambda}})^+$. We write $\chi_0' := \gamma (\chi_0)$.
\end{notation}
\begin{remark}\label{rmk-hanf}
  By Fact \ref{stab-op}, one can take $\chi_0 < H_1$. In that case also $\chi_0' < H_1$. For the sake of generality, we do \emph{not} require that $\chi_0$ be least with the property above.
\end{remark}

Recall \cite[5.21]{bv-sat-v3} that if $\BI$ is a $(\chi_0^+, \chi_0^+, \gamma (\chi_0))$-Morley sequence, then $\BI$ is $\chi$-convergent. We want to use this to relate average and forking:

\begin{defin}
  Let $M_0, M \in \Ksatp{(\chi_0')^+}$ be such that $M_0 \lea M$. Let $p \in \gS (M)$. We say that \emph{$p$ does not fork over $M_0$} if there exists $M_0' \in \K_{\chi_0'}$ such that $M_0' \lea M_0$ and $p$ does not $\chi_0'$-split over $M_0'$.
\end{defin}

We will use without comments:

\begin{fact}\label{forking-props}
  Forking has the following properties:

  \begin{enumerate}
    \item Invariance under isomorphisms and monotonicity: if $M_0 \lea M_1 \lea M_2$ are all $(\chi_0')^+$-saturated and $p \in \gS (M_2)$ does not fork over $M_0$, then $p \rest M_1$ does not fork over $M_0$ and $p$ does not fork over $M_1$.
    \item\label{forking-props-2} Set local character: if $M \in \Ksatp{(\chi_0')^+}$ and $p \in \gS (M)$, there exists $M_0 \in \Ksatp{(\chi_0')^+}$ of size $(\chi_0')^+$ such that $M_0 \lea M$ and $p$ does not fork over $M_0$.
    \item Transitivity: Assume $M_0 \lea M_1 \lea M_2$ are all $(\chi_0')^+$-saturated and $p \in \gS (M_2)$. If $p$ does not fork over $M_1$ and $p \rest M_1$ does not fork over $M_0$, then $p$ does not fork over $M_0$.
    \item\label{forking-props-4} Uniqueness: If $M_0 \lea M$ are all $(\chi_0')^+$-saturated and $p, q \in \gS (M)$ do not fork over $M_0$, then $p \rest M_0 = q \rest M_0$ implies $p = q$. Moreover $p$ does not $\lambda$-split over $M_0$ for any $\lambda \ge \left(\chi_0'\right)^+$.
    \item\label{forking-props-5} Local extension over saturated models: If $M_0 \lea M$ are both saturated, $\|M_0\| = \|M\| \ge (\chi_0')^+$, $p \in \gS (M_0)$, there exists $q \in \gS (M)$ such that $q$ extends $p$ and does not fork over $M_0$.
  \end{enumerate}
\end{fact}
\begin{proof}
  Use \cite[7.5]{indep-aec-apal}. The generator used is the one given by Proposition 7.4(2) there. For the moreover part of uniqueness, use \cite[4.2]{bgkv-apal} (and \cite[3.12]{bgkv-apal}).
\end{proof}

Note that the extension property need not hold in general. However if the class has no long splitting chains we have:

\begin{fact}\label{union-sat}
  If $\K$ has no long splitting chains in $\chi_0'$, then:

  \begin{enumerate}
  \item (\cite[8.9]{indep-aec-apal} or \cite[7.1]{ss-tame-jsl}) Forking has:
    \begin{enumerate}
      \item The extension property: If $M_0 \lea M$ are $(\chi_0')^+$-saturated and $p \in \gS (M_0)$, then there exists $q \in \gS (M)$ extending $p$ and not forking over $M_0$.
      \item The chain local character property: If $\seq{M_i : i < \delta}$ is an increasing chain of $(\chi_0')^+$-saturated models and $p \in \gS (\bigcup_{i < \delta} M_i)$, then there exists $i < \delta$ such that $p$ does not fork over $M_i$.
    \end{enumerate}
  \item \cite[6.9]{bv-sat-v3} For any $\lambda > (\chi_0')^+$, $\Ksatp{\lambda}$ is an AEC with $\LS (\Ksatp{\lambda}) = \lambda$.
  \end{enumerate}
\end{fact}

For notational convenience, we ``increase'' $\chi_0$:

\begin{notation}\label{notation-2}
  Let $\chi_1 := (\chi_0')^{++}$. Let $\chi_1' := \gamma (\chi_1)$.
\end{notation}

We obtain a characterization of forking that adds to those proven in \cite[Section 9]{indep-aec-apal}. A form of it already appears in \cite[IV.4.6]{shelahaecbook}. Again, we define more cardinal parameters:

\begin{notation}\label{notation-3}
  Let $\chi_2 := \beth_{\omega} (\chi_0)$.
\end{notation}
\begin{remark}\label{rmk-hanf-2}
  We have that $\chi_0 < \chi_0' < \chi_1 < \chi_1' < \chi_2$, and $\chi_2 < H_1$ if $\chi_0 < H_1$.
\end{remark}

\begin{lem}\label{forking-charact}
  Let $M_0, M$ be $\chi_2$-saturated with $M_0 \lea M$. Let $p \in \gS (M)$. The following are equivalent:

  \begin{enumerate}
    \item\label{forking-char-1} $p$ does not fork over $M_0$.
    \item\label{forking-char-2} $p \rest M_0$ has a nonforking extension to $\gS (M)$ and there exists $M_0' \lea M_0$ with $\|M_0'\| < \chi_2$ such that $p$ does not syntactically split over $M_0'$.
    \item\label{forking-char-3} $p \rest M_0$ has a nonforking extension to $\gS (M)$ and there exists $\mu \in [\chi_0^+, \chi_2)$ and $\BI$ a $(\mu, \mu, \gamma (\mu)^+)$-Morley sequence for $p$, with all the witnesses inside $M_0$, such that $\Av_{\gamma (\mu)} (\BI / M) = p$.
  \end{enumerate}
\end{lem}
\begin{remark}
  When $\K$ has no long splitting chains in $\chi_0'$, forking has the extension property (Fact \ref{union-sat}) so the first part of (\ref{forking-char-2}) and (\ref{forking-char-3}) always hold. However in Theorem \ref{ss-from-chainsat} we apply Lemma \ref{forking-charact} in the strictly stable case (i.e.\ $\K$ may only be \emph{stable} in $\chi_0'$ and not have no long splitting chains there).
\end{remark}

We recall more definitions and facts before giving the proof of Lemma \ref{forking-charact}:

\begin{fact}[V.A.1.12 in \cite{shelahaecbook2}]\label{ns-lc}
  If $p \in \gS (M)$ and $M$ is $\chi_0^+$-saturated, there exists $M_0 \in \K_{\le \chi_0}$ with $M_0 \lea M$ such that $p$ does not syntactically split over $M_0$.
\end{fact}

%% \begin{defin}[Definition 5.9 in \cite{bv-sat-v3}]
%%   A sequence $\BI$ is \emph{$\mu$-based on $M_0$} if for any $M$ with $M_0 \lea M$, $\Av_{\mu} (\BI / M)$ does not syntactically split over $M_0$ (when the average exists).
%% \end{defin}

%% \begin{fact}[Claim IV.1.23(2) in \cite{shelahaecbook} or see Lemma 5.10 in \cite{bv-sat-v3}]\label{based-fact-2}
%%   Let $\BI$ be a sequence and let $\BJ \subseteq \BI$ have size at least $\mu$. Then $\BI$ is $\mu$-based on any model $M_0 \lea \sea$ containing $\BJ$.
%% \end{fact}
%% \begin{fact}[Lemma 5.20 in \cite{bv-sat-v3}]\label{based-fact}
%%   Let $\BI$ be a $(\mu^+, \mu^+, \mu^+)$-Morley sequence over $M$ (for some type). If $\BI$ is $\mu$-convergent, then $\BI$ is $\mu$-based on $M$.
%% \end{fact}

\begin{fact}\label{average-splitting}
  Let $M_0 \lea M$ be both $(\chi_1')^+$-saturated. Let $\mu := \|M_0\|$. Let $p \in \gS (M)$ and let $\BI$ be a $(\mu^+, \mu^+, \gamma (\mu))$-Morley sequence for $p$ over $M_0$ with all the witnesses inside $M$. Then if $p$ does not syntactically split or does not fork over $M_0$, then $\Av_{\gamma (\mu)} (\BI / M) = p$.
\end{fact}
\begin{proof}
  For syntactic splitting, this is \cite[5.25]{bv-sat-v3}. The Lemma is actually more general and the proof of \cite[6.9]{bv-sat-v3} shows that this also works for forking. 
\end{proof}

\begin{proof}[Proof of Lemma \ref{forking-charact}] 
  Before starting, note that if $\mu < \chi_2$, then $\K$ is stable in $2^{{\mu + \chi_0}} < \chi_2$ by Fact \ref{stab-spectrum-2}. Thus there are unboundedly many stability cardinals below $\chi_2$, so we have ``enough space'' to build Morley sequences.

  \begin{itemize}
    \item \underline{(\ref{forking-char-1}) implies (\ref{forking-char-2})}:
      By Fact \ref{ns-lc}, we can find $M_0' \lea M_0$ such that $p \rest M_0$ does not syntactically split over $M_0'$ and $\|M_0'\| \le \chi_1$. Taking $M_0'$ bigger, we can assume $M_0'$ is $\chi_1$-saturated and $p \rest M_0$ does not fork over $M_0'$. Thus by transitivity $p$ does not fork over $M_0'$. Let $\BI$ be a $(\chi_1^+, (\chi_1')^+, (\chi_1')^+)$-Morley sequence for $p \rest M_0$ over $M_0'$ inside $M_0$. By \cite[5.21]{bv-sat-v3}, $\BI$ is $\chi_1'$-convergent. By \cite[5.20]{bv-sat-v3}, $\BI$ is $\chi_1'$-based on $M_0'$. Note also that $\BI$ is a $(\chi_1^+, (\chi_1')^+, (\chi_1')^+)$-Morley sequence for $p$ over $M_0'$ and by Fact \ref{average-splitting}, $\Av_{\chi_1'} (\BI / M_0) = p$ so as $\BI$ is $\chi_1'$-based on $M_0'$, $p$ does not syntactically split over $M_0'$.
    \item \underline{(\ref{forking-char-2}) implies (\ref{forking-char-3})}:
      As in the proof of (\ref{forking-char-1}) implies (\ref{forking-char-2}) (except $\chi_1$ could be bigger).
    \item \underline{(\ref{forking-char-3}) implies (\ref{forking-char-2})}:
      By Fact \cite[5.21]{bv-sat-v3}, $\BI$ is $\gamma (\mu)$-convergent. Pick any $\BJ \subseteq \BI$ of length $\gamma (\mu)$ and use \cite[5.10]{bv-sat-v3} to find $M_0' \lea M_0$ of size $\gamma (\mu)$ such that $\BJ$ is $\gamma (\mu)$-based on $M_0'$. Since also $\BJ$ is $\gamma (\mu)$-convergent, we have that $\BI$ is $\gamma (\mu)$-based on $M_0'$. Thus $\Av_{\gamma (\mu)} (\BI / M) = p$ does not syntactically split over $M_0'$.
    \item \underline{(\ref{forking-char-2}) implies (\ref{forking-char-1})}:
      Without loss of generality, we can choose $M_0'$ to be such that $p \rest M_0$ also does not fork over $M_0'$. Let $\mu := \|M_0'\| + \chi_0$. Build a $(\mu^+, \mu^+, \gamma (\mu))$-Morley sequence $\BI$ for $p$ over $M_0'$ inside $M_0$. If $q$ is the nonforking extension of $p \rest M_0$ to $M$, then $\BI$ is also a Morley sequence for $q$ over $M_0'$ so by the proof of (\ref{forking-char-1}) implies (\ref{forking-char-2}) we must have $\Av_{\gamma (\mu)} (\BI / M) = q$, but also $\Av_{\gamma (\mu)} (\BI / M) = p$, since $p$ does not syntactically split over $M_0'$ (Fact \ref{average-splitting}). Thus $p = q$.
  \end{itemize}
\end{proof}

The next result is a version of \cite[III.3.10]{shelahfobook} in our context. It is implicit in the proof of \cite[5.27]{bv-sat-v3}.

\begin{lem}\label{saturation-charact}
  Let $M \in \Ksatp{\chi_2}$. Let $\lambda \ge \chi_2$ be such that $\K$ is stable in unboundedly many $\mu < \lambda$. The following are equivalent.

  \begin{enumerate}
    \item\label{sat-cond-1} $M$ is $\lambda$-saturated.
    \item\label{sat-cond-2} If $q \in \gS (M)$ is not algebraic and does not syntactically split over $M_0 \lea M$ with $\|M_0\| < \chi_2$, there exists a $((\|M_0\| + \chi_0)^+, (\|M_0\| + \chi_0)^+, \lambda)$-Morley sequence for $p$ over $M_0$ inside $M$.
  \end{enumerate}
\end{lem}

%% We use one more fact in the proof, telling us when the average is realized by an element of the sequence.

%% \begin{fact}[Lemma 5.6 in \cite{bv-sat-v3}]\label{realize-counting}
%%   Let $\BI$ be a sequence and let $\mu$ be a cardinal such that $\BI \ge \mu$. Let $M \lea \sea$ and let $p \in \gS (M)$ be such that $\Av_\mu (\BI / M) = p$. If $|\BI| > \mu + |\gS (M)|$, then there exists $b \in \BI$ realizing $p$.
%% \end{fact}

\begin{proof}
  (\ref{sat-cond-1}) implies (\ref{sat-cond-2}) is trivial using saturation. Now assume (\ref{sat-cond-2}). Let $p \in \gS (N)$, $\|N\| < \lambda$, $N \lea M$. We show that $p$ is realized in $M$. Let $q \in \gS (M)$ extend $p$. If $q$ is algebraic, we are done so assume it is not. Let $M_0 \lea M$ have size $(\chi_1')^+$ such that $q$ does not fork over $M_0$. By Lemma \ref{forking-charact}, we can increase $M_0$ if necessary so that $q$ does not syntactically split over $M_0$ and $\mu := \|M_0\| \ge \chi_0$. Now by (\ref{sat-cond-2}), there exists a $(\mu^+, \mu^+, \lambda)$-Morley sequence $\BI$ for $q$ over $M_0$ inside $M$. Now by Fact \ref{average-splitting}, $\Av_{\gamma (\mu)} (\BI / M) = q$. Thus $\Av_{\gamma (\mu)} (\BI / N) = p$. By \cite[5.6]{bv-sat-v3} and the hypothesis of stability in unboundedly many cardinals below $\lambda$, $p$ is realized by an element of $\BI$ and hence by an element of $M$.
\end{proof}

We end by showing that if high-enough limit models are sufficiently saturated, then no long splitting chains holds. A similar argument already appears in the proof of \cite[IV.4.10]{shelahaecbook}. We start with a more local version, 

\begin{lem}\label{forking-chain-lc}
  Let $\lambda \ge \chi_2$. Let $\delta < \lambda^+$ be a limit ordinal and let $\seq{M_i : i < \delta}$ be an increasing chain of saturated models in $\K_\lambda$. Let $M_\delta := \bigcup_{i < \delta} M_i$. If $M_\delta$ is $\chi_2$-saturated, then for any $p \in \gS (M_\delta)$, there exists $i < \delta$ such that $p$ does not fork over $M_i$.
\end{lem}
\begin{proof}
  Without loss of generality, $\delta$ is regular. If $\delta \ge \chi_2$, by set local character (Fact \ref{forking-props}(\ref{forking-props-2})), there exists $M_0'$ of size $\chi_1$ such that $p$ does not fork over $M_0'$ and $M_0' \lea M_\delta$, so pick $i < \delta$ such that $M_0' \lea M_i$ and use monotonicity. 

Now assume $\delta < \chi_2$. By assumption, we have that $M_\delta$ is $\chi_2$-saturated. We also have that $p$ does not fork over $M_\delta$ (by set local character) so by Lemma \ref{forking-charact}, there exists $\mu \in [\chi_0^+, \chi_2)$ and $\BI$ a $(\mu, \mu, \gamma (\mu)^+)$-Morley sequence for $p$ with all the witnesses inside $M_\delta$ such that $\Av_{\gamma (\mu)} (\BI / M_\delta) = p$. Since $M_\delta$ is $\chi_2$-saturated (and there are unboundedly many stability cardinals below $\chi_2$), we can increase $\BI$ if necessary to assume that $|\BI| \ge \chi_2$. Write $\BI_i := |M_i| \cap \BI$. Since $\delta < \chi_2$, there must exists $i < \delta$ such that $|\BI_i| \ge \chi_2$. Note that $\BI_i$ is a $(\mu, \mu, \chi_2)$-Morley sequence for $p$. Because $\BI$ is $\gamma (\mu)$-convergent and $|\BI_i| \ge \chi_2 > \gamma (\mu)$, $\Av_{\gamma (\mu)} (\BI_i / M_\delta) = p$. Letting $M' \gea M_\delta$ be a saturated model of size $\lambda$ and using local extension over saturated models (Fact \ref{forking-props}(\ref{forking-props-5})), $p \rest M_i$ has a nonforking extension to $\gS (M')$ and hence to $\gS (M_\delta)$. By Lemma \ref{forking-charact}, $p$ does not fork over $M_i$, as desired.
\end{proof}

\begin{thm}\label{ss-from-chainsat}
  Assume that $\K$ has a monster model, is $\LS (\K)$-tame, and stable in some cardinal greater than or equal to $\LS (\K)$.

  Let $\chi_0 \ge \LS (\K)$ be such that $\K$ does not have the $\LS (\K)$-order property of length $\chi_0$, and let $\chi_2 := \beth_\omega (\chi_0)$. Let $\lambda \ge \chi_2$ be such that $\K$ is stable in $\lambda$ and there exists a saturated model of cardinality $\lambda$. If every limit model of cardinality $\lambda$ is $\chi_2$-saturated, then $\K$ has no long splitting chains in $\lambda$.
\end{thm}
\begin{proof}
  Let $\K'$ be $\Ksatpnobf{\chi_2}_\lambda$ ordered by being equal or universal over. Note that, by stability in $\lambda$, $\K'$ is a skeleton of $\K_\lambda$ (see Definition \ref{skel-def}). Moreover since every limit model of cardinality $\lambda$ is $\chi_2$-saturated, for any limit $\delta < \lambda^+$, one can build an increasing continuous chain $\seq{M_i : i \le \delta}$ in $\K_\lambda$ such that for all $i \le \delta$, $M_i$ is $\chi_2$-saturated and (when $i < \delta$) $M_{i + 1}$ is universal over $M_i$. Therefore limit models exist in $\K'$, so the assumptions of Fact \ref{skeleton-facts}(\ref{skel-2b}) are satisfied. So it is enough to see that $\K'$ (not $\K$) has no long splitting chains in $\lambda$.
  
  Let $\delta < \lambda^+$ be limit and let $\seq{M_i : i < \delta}$ be an increasing chain of models in $\K'$, with $M_{i + 1}$ universal over $M_i$ for all $i < \delta$. Let $M_\delta := \bigcup_{i < \delta} M_i$. By assumption, $M_\delta$ is $\chi_2$-saturated. By uniqueness of limit models of the same length, we can assume without loss of generality that $M_{i + 1}$ is saturated for all $i < \delta$.

  Let $p \in \gS (M_\delta)$. By Lemma \ref{forking-chain-lc} (applied to $\seq{M_{i + 1} : i < \delta}$), there exists $i < \delta$ such that $p$ does not fork over $M_i$. By the moreover part of Fact \ref{forking-props}(\ref{forking-props-4}), $p$ does not $\lambda$-split over $M_i$, as desired.
\end{proof}

\section{No long splitting chains implies solvability}\label{solvability-sec}

From now on we assume no long splitting chains:

\begin{hypothesis}\label{ns-hyp} \
  \begin{enumerate}
    \item Hypothesis \ref{global-nf-hyp}, and we fix cardinals $\chi_0 < \chi_0' < \chi_1 < \chi_1' < \chi_2$ as defined in Notation \ref{notation-1}, \ref{notation-2}, and \ref{notation-3}. Note that by Fact \ref{stab-spectrum-2} $\K$ is stable in $\chi_0'$.
    \item $\K$ has no long splitting chains in $\chi_0'$.
  \end{enumerate}

  In Notation \ref{notation-4}, we will define another cardinal $\chi$ with $\chi_2 < \chi$. If $\chi_0 < H_1$, we will also have that $\chi < H_1$.
\end{hypothesis}

Note that no long splitting chains in $\chi_0'$ and stability in $\chi_0'$ implies (Fact \ref{ss-stable}) that $\K$ is stable in all $\lambda \ge \chi_0'$. Further, forking is well-behaved in the sense of Fact \ref{union-sat}. This implies that Morley sequences are closed under unions (here we use that they are indexed by arbitrary linear orders, as opposed to just well-orderings). Recall  that we say $\BI \smallfrown \seq{N_i : i \le \delta}$ is a Morley sequence when $\BI$ is a sequence of elements and the $N_i$'s are an increasing chain of sufficiently saturated models witnessing that $\BI$ is Morley, see \cite[5.14]{bv-sat-v3} for the precise definition.

\begin{lem}\label{morley-union}
  Let $\seq{I_\alpha : \alpha < \delta}$ be an increasing (with respect to substructure) sequence of linear orders and let $I_\delta := \bigcup_{\alpha < \delta} I_\alpha$. Let $M_0, M$ be $\chi_2$-saturated such that $M_0 \lea M$. Let $\mu_0, \mu_1, \mu_2$ be such that $\chi_2 < \mu_0 \le \mu_1 \le \mu_2$, $p \in \gS (M)$ and for $\alpha < \delta$, let $\BI_\alpha := \seq{a_i : i \in I_\alpha}$ together with $\seq{N_i^\alpha : i \in I_\alpha}$ be $(\mu_0, \mu_1, \mu_2)$-Morley for $p$ over $M_0$, with $N_i^\alpha \lea N_i^{\beta} \lea M$ for all $\alpha \le \beta < \delta$ and $i \in I_\alpha$. For $i \in I_\alpha$, let $N_i^\delta := \bigcup_{\beta \in [\alpha, \delta)} N_i^{\beta}$. Let $\BI_\delta := \seq{a_i : i \in I_\delta}$.

  If $p$ does not fork over $M_0$, then $\BI_\delta \smallfrown \seq{N_i^\delta : i \in I_\delta}$ is $(\mu_0, \mu_1, \mu_2)$-Morley for $p$ over $M_0$.
\end{lem}
\begin{proof}
  By Lemma \ref{forking-charact}, $p$ does not syntactically split over $M_0$. Therefore the only problematic clauses in \cite[Definition 5.14]{bv-sat-v3} are (4) and (7). Let's check (4): let $i \in I_\delta$. By hypothesis, $\ba_i$ realizes $p \rest N_i^\alpha$ for all sufficiently high $\alpha < \delta$. By local character of forking, there exists $\alpha < \delta$ such that $\gtp (\ba_i / N_i^\delta)$ does not fork over $N_i^\alpha$. Since $\gtp (\ba_i / N_i^\delta) \rest N_i^\alpha = p \rest N_i^\alpha$ and $p$ does not fork over $M_0 \lea N_i^\alpha$, we must have by uniqueness that $p \rest N_i^\delta = \gtp (\ba_i / N_i^\delta)$. The proof of (7) is similar.
\end{proof}

For convenience, we make $\chi_2$ even bigger:

\begin{notation}\label{notation-4}
  Let $\chi := \gamma (\chi_2)$ (recall from Notation \ref{notation-1} that $\gamma (\chi_2) = \left(2^{2^{\chi_2}}\right)^+$). A Morley sequence means a $(\chi_2^+, \chi_2^+, \chi)$-Morley sequence.
\end{notation}
\begin{remark}\label{rmk-hanf-3}
  By Remark \ref{rmk-hanf-2}, we still have $\chi < H_1$ if $\chi_0 < H_1$.
\end{remark}

We are finally in a position to prove solvability (in fact even uniform solvability). We will use condition (\ref{equiv-cond-3}) in Lemma \ref{solvability-equiv}. 

\begin{defin}
  We define a class of models $K'$ and a binary relation $\leap{\K'}$ on $K'$ (and write $\K' := (K', \leap{\K'})$) as follows.

  \begin{itemize}
    \item $K'$ is a class of $\tau' := \tau (\K')$-structures, with:

      $$
      \tau' := \tau (\K) \cup \{N_0, N, F, R\}
      $$

      where:

      \begin{itemize}
        \item $N_0$ and $R$ are binary relations symbols.
        \item $N$ is a ternary relation symbol.
        \item $F$ is a binary function symbol.
      \end{itemize}
    \item A $\tau'$-structure $M$ is in $K'$ if and only if:

      \begin{enumerate}
        \item $M \rest \tau (\K) \in \Ksatp{\chi}$.
        \item $R^M$ is a linear ordering of $|M|$. We write $I$ for this linear ordering.
        \item\label{cond-3-k'} For\footnote{For a binary relation $Q$ we write $Q (a)$ for $\{b \mid Q (a, b)\}$, similarly for a ternary relation.} all $a \in |M|$ and all $i \in I$, $N^M (a, i) \lea M \rest \tau (\K)$ (where we see $N^M (a, i)$ as an $\tau (\K)$-structure; in particular, $N^M (a, i) \in \K$; it will follow from (\ref{morley-seq-cond}) that the $N^M (a, i)$'s are increasing with $i$, $N_0^M (a) \lea N^M (a, i)$, and $N_0^M (a)$ is saturated of size $\chi_2$.
        \item\label{cond-4-k'} There exists a map $a \mapsto p_a$ from $|M|$ onto the non-algebraic Galois types (of length one) over $M \rest \tau (\K)$ such that for all $a \in |M|$:
          \begin{enumerate}
          \item $p_a$ does not fork\footnote{Note that by Lemma \ref{forking-charact} this also implies that it does not syntactically split over some $M_0' \lea N_0^M (a)$ with $\|M_0'\| < \chi_2$.} over $N_0^M (a)$.
            \item\label{morley-seq-cond} $\seq{F^M (a, i) : i \in I} \smallfrown \seq{N^M (a, i) : i \in I}$ is a Morley sequence for $p_a$ over $N_0^M (a)$. 
          \end{enumerate}
      \end{enumerate}
    \item $M \leap{K'} M'$ if and only if:
      \begin{enumerate}
        \item $M \subseteq M'$.
        \item $M \rest \tau (\K) \lea M' \rest \tau (\K)$.
        \item For all $a \in |M|$, $N_0^M (a) = N_0^{M'} (a)$.
      \end{enumerate}
    \end{itemize}
\end{defin}

We show in Lemma \ref{k'-aec} that  $\K'$ is an AEC, but first let us see that this suffices: 

\begin{lem}\label{k'-expansion}
  Let $\lambda \ge \chi$.
  \begin{enumerate}
    \item If $M \in \K_\lambda$ is saturated, then there exists an expansion $M'$ of $M$ to $\tau'$ such that $M' \in \K'$. 
    \item If $M' \in \K'$ has size $\lambda$, then $M' \rest \tau (\K)$ is saturated.
  \end{enumerate}
\end{lem}
\begin{proof} \
  \begin{enumerate}
    \item Let $R^{M'}$ be a well-ordering of $|M|$ of type $\lambda$. Identify $|M|$ with $\lambda$. By stability, we can fix a bijection $p \mapsto a_p$ from $\gS (M)$ onto $|M|$. For each $p \in \gS (M)$ which is not algebraic, fix $N_p \lea M$ saturated such that $p$ does not fork over $N_p$ and $\|N_p\| = \chi_2$. Then use saturation to build $\seq{a_p^i : i < \lambda} \smallfrown \seq{N_p^i : i < \lambda}$ Morley for $p$ over $N_p$ (inside $M$). Let $N_0^{M'} (a_p) := N_p$, $N^{M'} (a_p, i) := N_p^i$, $F^{M'} (a, i) := a_p^i$. For $p$ algebraic, pick $p_0 \in \gS (M)$ nonalgebraic and let $N_0^{M'} (a_p) := N_0^{M'} (a_{p_0})$, $N^{M'} (a_{p_0}) := N^{M'} (a_{p_0})$, $F^{M'} (a_{p}) := F^{M'} (a_{p_0})$.
    \item By Lemma \ref{saturation-charact}.
  \end{enumerate}
\end{proof}

\begin{lem}\label{k'-aec}
  $\K'$ is an AEC with $\LS (\K') = \chi$.
\end{lem}
\begin{proof}
  It is straightforward to check that $\K'$ is an abstract class with coherence. Moreover:

  \begin{itemize}
    \item \underline{$\K'$ satisfies the chain axioms}: Let $\seq{M_i : i < \delta}$ be increasing in $\K'$. Let $M_\delta := \bigcup_{i < \delta} M_i$. 
      \begin{itemize}
      \item $M_0 \leap{\K'} M_\delta$, and if $N \geap{\K'} M_i$ for all $i < \delta$, then $N \geap{\K'} M_\delta$: Straightforward.
      \item $M_\delta \in \K'$: $M_\delta \rest \tau (\K)$ is $\chi$-saturated by Fact \ref{union-sat}. Moreover, $R^{M_\delta}$ is clearly a linear ordering of $M_\delta$. Write $I_i$ for the linear ordering $(M_i, R_i)$. Condition \ref{cond-3-k'} in the definition of $\K'$ is also easily checked. We now check Condition \ref{cond-4-k'}.

        Let $a \in |M_\delta|$. Fix $i < \delta$ such that $a \in |M_i|$. Without loss of generality, $i = 0$. By hypothesis, for each $i < \delta$, there exists $p_a^i \in \gS (M_i \rest \tau (\K))$ not algebraic such that $\seq{F^{M_i} (a, j) \mid j \in I_i} \smallfrown \seq{N^{M_i} (a, j) \mid j \in I_i}$ is a Morley sequence for $p_a^i$ over $N_0^{M_i} (a) = N_0^{M_0} (a)$. Clearly, $p_a^i \rest N_0^{M_0} (a) = p_a^0 \rest N_0^{M_0} (a)$ for all $i < \delta$. Moreover by assumption $p_a^i$ does not fork over $N_0^{M_0}$. Thus for all $i < j < \delta$, $p_a^j \rest M_i = p_a^i \rest M_i$. By extension and uniqueness, there exists $p_a \in \gS (M_\delta \rest \tau (\K))$ that does not fork over $N_0^{M_0} (a)$ and we have $p_a \rest M_i = p_a^i$ for all $i < \delta$. Now by Lemma \ref{morley-union}, $\seq{F^{M_\delta} (a, j) \mid j \in I_\delta} \smallfrown \seq{N^{M_\delta} (a, j) \mid j \in I_\delta}$ is a Morley sequence for $p_a$ over $N_0^{M_0} (a)$. 

        Moreover, the map $a \mapsto p_a$ is onto the nonalgebraic Galois types over $M_\delta \rest \tau (\K)$: let $p \in \gS (M_\delta \rest \tau (\K))$ be nonalgebraic. Then there exists $i < \delta$ such that $p$ does not fork over $M_i$. Let $a \in |M_i|$ be such that $\seq{F^{M_i} (a, j) \mid j \in I_i} \smallfrown \seq{N^{M_i} (a, j) \mid j \in I_i}$ is a Morley sequence for $p \rest M_i$ over $N_0^{M_i} (a)$. It is easy to check it is also a Morley sequence for $p$ over $N_0^{M_i} (a)$. By uniqueness of the nonforking extension, we get that the extended Morley sequence is also Morley for $p$, as desired.
      \end{itemize}
    \item \underline{$\LS (\K') = \chi$}: An easy closure argument.
  \end{itemize}
\end{proof}

\begin{thm}\label{strong-solvable}
  $\K$ is uniformly $(\chi, \chi)$-solvable.
\end{thm}
\begin{proof}
  By Lemma \ref{k'-aec}, $\K'$ is an AEC with $\LS (\K') = \chi$. Now combine Lemma \ref{k'-expansion} and Lemma \ref{solvability-equiv}. Note that by Fact \ref{union-sat}, for each $\lambda \ge \chi$ there is a saturated model of size $\lambda$, and it is also a superlimit.
\end{proof}

For the convenience of the reader, we give a more quotable version of Theorem \ref{strong-solvable}. For the next results, we drop Hypothesis \ref{ns-hyp}.

\begin{thm}\label{strong-solvable-thm}
  Assume that $\K$ has a monster model, is $\LS (\K)$-tame, and is stable in some cardinal greater than or equal to $\LS (\K)$. There exists $\chi < H_1$ such that for any $\mu \ge \chi$, if $\K$ is stable in $\mu$ and has no long splitting chains in $\mu$ then $\K$ is uniformly $(\mu', \mu')$-solvable, where $\mu' := \left(\beth_{\omega + 2} (\mu)\right)^+$.
\end{thm}
\begin{proof}
  Hypothesis \ref{global-nf-hyp} holds. Let $\chi < H_1$ be such that $\K$ does not have the $\LS (\K)$-order property of length $\chi$ (see Fact \ref{stab-op}).

  Let $\mu \ge \chi$ be such that $\K$ is stable in $\mu$ and has no long splitting chains in $\mu$. We apply Theorem \ref{strong-solvable} by letting $\chi_0$ in Notation \ref{notation-1} stand for $\mu$ here. By Fact \ref{ss-stable}, $\K$ is stable in $\mu_1$ and has no long splitting chains in $\mu_1$ for every $\mu_1 \ge \mu$, thus Hypothesis \ref{ns-hyp} holds. Moreover $\chi_2$ in Notation \ref{notation-3} corresponds to $\beth_\omega (\mu)$ here, and $\chi$ in Notation \ref{notation-4} corresponds to $\mu'$ here. Thus the result follows from Theorem \ref{strong-solvable}.
\end{proof}

\begin{cor}
  Assume that $\K$ has a monster model and is $\LS (\K)$-tame. If there exists $\mu < H_1$ such that $\K$ is stable in $\mu$ and has no long splitting chains in $\mu$, then there exists $\mu' < H_1$ such that $\K$ is uniformly $(\mu', \mu')$-solvable.
\end{cor}
\begin{proof}
  Let $\mu < H_1$ be such that $\K$ is stable in $\mu$ and has no long splitting chains in $\mu$. Fix $\chi < H_1$ as given by Theorem \ref{strong-solvable-thm}. Without loss of generality, $\mu \le \chi$. By Fact \ref{ss-stable}, $\K$ is stable in $\chi$ and has no long splitting chains in $\chi$, so apply the conclusion of Theorem \ref{strong-solvable-thm}.
\end{proof}

\section{Superstability below the Hanf number}\label{hanf-section}

In this section, we prove the main corollary. In fact, we prove a stronger version that instead of asking for the properties to hold on a tail asks for them to hold only in a single high-enough cardinal. Toward this end, we start by explaining why no long splitting chains follows from categoricity in a high-enough cardinal. In fact, categoricity can be replaced by solvability. All the ingredients for this result are contained in \cite{shvi635} and this specific form has only appeared recently \cite[Theorem 3]{shvi-notes-v3-toappear}. Note also that Shelah states a similar result in \cite[5.5]{sh394} but his definition of superstability is different.

\begin{fact}[The ZFC Shelah-Villaveces theorem]\label{ns-lc-fact}
  Let $\K$ be an AEC with arbitrarily large models and amalgamation\footnote{In \cite{shvi635}, this is replaced by the generalized continuum hypothesis (GCH).} in $\LS (\K)$. Let $\lambda > \LS (\K)$ be such that $\K_{<\lambda}$ has no maximal models. If $\K$ is $(\lambda, \LS (\K))$-solvable, then $\K$ is stable in $\LS (\K)$ and has no long splitting chains in $\LS (\K)$.
\end{fact}

\begin{cor}\label{ns-lc-cor}
  Let $\K$ be an AEC with a monster model. Let $\lambda > \LS (\K)$. If $\K$ is categorical in $\lambda$, then $\K$ is stable in $\mu$ and has no long splitting chains in $\mu$ for all $\mu \in [\LS (\K), \lambda)$.
\end{cor}
\begin{proof}
  By Fact \ref{ns-lc-fact} applied to $\K_{\ge \mu}$ for each $\mu \in [\LS (\K), \lambda)$. Note that, since $\K$ has arbitrarily large models, categoricity in $\lambda$ implies $(\lambda, \LS (\K))$-solvability.
\end{proof}

We conclude that solvability is equivalent to superstability in the first-order case:

\begin{cor}\label{solvability-fo}
  Let $T$ be a first-order theory and let $\K$ be the AEC of models of $T$ ordered by elementary substructure. Let $\mu \ge |T|$. The following are equivalent:

  \begin{enumerate}
  \item\label{ss-cond-1} $T$ is stable in all $\lambda \ge \mu$.
  \item\label{ss-cond-2} $\K$ is $(\lambda, \mu)$-solvable, for some $\lambda > \mu$.
  \item\label{ss-cond-3} $\K$ is uniformly $(\mu, \mu)$-solvable.
  \end{enumerate}
\end{cor}
\begin{proof}[Proof sketch]
  (\ref{ss-cond-3}) implies (\ref{ss-cond-2}) is trivial. (\ref{ss-cond-2}) implies (\ref{ss-cond-1}) is by Corollary \ref{ns-lc-cor} together with Fact \ref{ss-stable}). Finally, (\ref{ss-cond-1}) implies (\ref{ss-cond-3}) is as in the proof of Theorem \ref{strong-solvable-thm}.
\end{proof}

We can also use the ZFC Shelah-Villaveces theorem to prove the following interesting result, showing that the solvability spectrum satisfies an analog of Shelah's categoricity conjecture in tame AECs (Shelah asks what the behavior of the solvability spectrum should be in \cite[Question N.4.4]{shelahaecbook}).

\begin{thm}\label{solv-transfer}
  Assume that $\K$ has a monster model and is $\LS (\K)$-tame. There exists $\chi < H_1$ such that for any $\mu \ge \chi$, if $\K$ is $(\lambda, \mu)$-solvable for \emph{some} $\lambda > \mu$, then $\K$ is uniformly $(\mu', \mu')$-solvable, where $\mu' := \left(\beth_{\omega + 2} (\mu)\right)^+$.
\end{thm}
\begin{proof}
  Let $\chi < H_1$ be as given by Theorem \ref{strong-solvable-thm}. Let $\mu \ge \chi$ and fix $\lambda > \mu$ such that $\K$ is solvable in $\lambda$. By Fact \ref{ns-lc-fact}, $\K$ is stable in $\mu$ and has no long splitting chains in $\mu$. Now apply Theorem \ref{strong-solvable-thm}.
\end{proof}

We are now ready to prove the stronger version of the main corollary where the properties hold only in a single high-enough cardinal below $H_1$ (but the cardinal may be different for each property).

\begin{cor}\label{main-cor-unbounded}
  Assume that $\K$ has a monster model, is $\LS (\K)$-tame, and is stable in some cardinal greater than or equal to $\LS (\K)$. Then there exists $\chi \in (\LS (\K), H_1)$ such that the following are equivalent:

  \begin{itemize}
  \item[$(\ref{sc0-1})^-$]\label{ssm-1} For some $\lambda_1 \in [\chi, H_1)$, $\K$ is stable in $\lambda_1$ and has no long splitting chains in $\lambda_1$.

  \item[$(\ref{sc0-2})^-$]\label{ssm-2} For some $\lambda_2 \in [\chi, H_1)$, there is a good $\lambda_2$-frame on a skeleton of $\K_{\lambda_2}$.

  \item[$(\ref{sc0-3})^-$]\label{ssm-3} For some $\lambda_3 \in [\chi, H_1)$, $\K$ has a unique limit model of cardinality $\lambda_3$.

  \item[$(\ref{sc0-4})^-$]\label{ssm-4} For some $\lambda_4 \in [\chi, H_1)$, $\K$ is stable in $\lambda_4$ and has a superlimit model of cardinality $\lambda_4$.

  \item[$(\ref{sc0-5})^-$]\label{ssm-5} For some $\lambda_5 \in [\chi, H_1)$, the union of any increasing chain of $\lambda_5$-saturated models is $\lambda_5$-saturated.

  \item[$(\ref{sc0-6})^-$]\label{ssm-6} For some $\lambda_6 \in [\chi, H_1)$, for some $\mu < \lambda_6$, $\K$ is $(\lambda_6, \mu)$-solvable.
  \end{itemize}
\end{cor}
\begin{remark}\label{type-full-rmk}
  In $(\ref{sc0-2})^-$, we do \emph{not} assume that the good frame is type-full (i.e.\ it may be that there exists some nonalgebraic types which are not basic, so fork over their domain). However if $(\ref{sc0-1})^-$ holds, then the proof of $(\ref{sc0-1})^-$ implies $(\ref{sc0-2})^-$ (Fact \ref{good-frame-existence}) actually builds a \emph{type-full} frame. Therefore, in the presence of tameness, the existence of a good frame implies the existence of a \emph{type-full} good frame (in a potentially much higher cardinal, and over a different class).
\end{remark}

\begin{proof}[Proof of Corollary \ref{main-cor-unbounded}]
  By Fact \ref{stab-spectrum}, $\K$ does not have the $\LS (\K)$-order property. By Fact \ref{stab-op}, there exists $\chi_0 < H_1$ such that $\K$ does not have the $\LS (\K)$-order property of length $\chi_0$. Let $\chi := \beth_{\omega}\left(\chi_0 + \LS (\K)\right)$.

  We will use the following auxiliary condition, which is a weakening of $(\ref{sc0-3})^-$ (the problem is that we do not quite know that $(\ref{sc0-5})^{-}$ implies $(\ref{sc0-3})^-$ as $\K$ might not be stable in $\lambda_5$):

  \begin{itemize}
      \item[$(\ref{sc0-3})^\ast$] For some $\lambda_3^\ast \in [\chi, H_1)$, $\K$ is stable in $\lambda_3^\ast$, has a saturated model of cardinality $\lambda_3^\ast$, and every limit model of cardinality $\lambda_3^\ast$ is $\chi$-saturated.
  \end{itemize}

  We will prove the following claims, which put together give us what we want:

  \underline{Claim 1}: $(\ref{sc0-1})^- \Leftrightarrow (\ref{sc0-6})^-$.

  \underline{Claim 2}: $(\ref{sc0-3})^\ast \Rightarrow (\ref{sc0-1})^-$.
  
  \underline{Claim 3}: For $\ell \in \{1,2, 3, 4, 5\}$, $(\ell)^- \Rightarrow (\ref{sc0-3})^\ast$.

  \underline{Proof of Claim 1}: By Theorem \ref{strong-solvable-thm} and Fact \ref{ns-lc-fact}. $\dagger_{\text{Claim 1}}$

  \underline{Proof of Claim 2}: This is Theorem \ref{ss-from-chainsat}, where $\chi_2$ there stands for $\chi$ here. $\dagger_{\text{Claim 2}}$

  \underline{Proof of Claim 3}: It is enough to prove the following subclaims:
  \begin{itemize}
    \item[] \underline{Subclaim 1}: $(\ref{sc0-1})^- \Rightarrow (\ref{sc0-2})^- \Rightarrow (\ref{sc0-3})^-$.
      
    \item[] \underline{Subclaim 2}: $(\ref{sc0-4})^- \Rightarrow (\ref{sc0-3})^-$.

    \item[] \underline{Subclaim 3}: $(\ref{sc0-3})^- \Rightarrow (\ref{sc0-3})^\ast$.

    \item[] \underline{Subclaim 4}: $(\ref{sc0-5})^- \Rightarrow (\ref{sc0-3})^\ast$.

  \end{itemize}
  \begin{itemize}
  \item[] \underline{Proof of Subclaim 1}: By Fact \ref{good-frame-existence}. $\dagger_{\text{Subclaim 1}}$
  \item[] \underline{Proof of Subclaim 2}: By Fact \ref{local-implications}(\ref{local-4}). $\dagger_{\text{Subclaim 2}}$
  \item[] \underline{Proof of Subclaim 3}: By Fact \ref{local-implications}(\ref{local-3}). $\dagger_{\text{Subclaim 3}}$
  \item[] \underline{Proof of Subclaim 4}: Let $\lambda_3^\ast \in [\lambda_5, H_1)$ be a regular stability cardinal. Then $\K$ has a saturated model of cardinality $\lambda_3^\ast$, and from $(\ref{sc0-5})^-$ it is easy to see that any limit model of cardinality $\lambda_3^\ast$ is $\lambda_5$-saturated, hence $\chi$-saturated. $\dagger_{\text{Subclaim 4}}$
  \end{itemize}
\end{proof}

We can now prove the main result of this paper (Corollary \ref{main-cor}):

\begin{proof}[Proof of Corollary \ref{main-cor}]
  Let $\chi$ be as given by Corollary \ref{main-cor-unbounded}. By Fact \ref{stab-spectrum}, there exists unboundedly-many regular stability cardinals in $(\chi, H_1)$. This implies that for $\ell \in \{1, 2, 3, 4, 5, 6\}$, $(\ell)$ (from Corollary \ref{main-cor}) implies $(\ell)^-$ (from Corollary \ref{main-cor-unbounded}). Moreover $(\ref{sc0-1})^-$ implies both (\ref{sc0-1}) and (\ref{sc0-7}) by Fact \ref{ss-stable}. Since Corollary \ref{main-cor-unbounded} tells us that $(\ell_1)^-$ is equivalent to $(\ell_2)^-$ for $\ell_1, \ell_2 \in \{1, 2, 3, 4, 5 ,6\}$, it follows that $(\ell_1)$ is equivalent to $(\ell_2)$ as well, and (\ref{sc0-7}) is implied by any of these conditions. 
\end{proof}

\begin{question}\label{superlimit-question}
  Is stability in $\lambda_4$ needed in condition $(\ref{sc0-4})^-$ of Corollary \ref{main-cor-unbounded}? That is, can one replace the condition with:

  \begin{itemize}
    \item[$(\ref{sc0-4})^{--}$] For some $\lambda_4 \in [\chi, \theta)$, $\K$ has a superlimit model of cardinality $\lambda_4$.
  \end{itemize}
\end{question}

The answer is positive when $\K$ is an elementary class \cite[ 3.1]{sh868}.

\section{Future work}\label{end-section}

While we managed to prove that some analogs of the conditions in Fact \ref{fo-superstab} are equivalent, much remains to be done.

For example, one may want to make precise what the analog to (5) and (6) in \ref{fo-superstab} should be in tame AECs. One possible definition for (6) could be:

\begin{defin} Let $\lambda, \mu > \LS (\K)$.  We say that $\K$ has the \emph{$(\lambda,\mu)$-tree property} provided there exists $\{p_n(\x;\y_n) \mid n<\omega\}$ Galois-types over models of size less than $\mu$ and $\{M_\eta \mid \eta\in \sq{\leq\omega}\lambda\}$ such that for all $n<\omega, \nu\in\sq{n}\lambda$ and every $\eta\in\sq{\omega}\lambda$:
\[
\langle M_\eta,M_\nu\rangle\models p_n \iff \nu \text{ is an initial segment of }\eta.  
\]

We say that $\K$ has the \emph{tree property} if it has it for all high-enough $\mu$ and all high-enough $\lambda$ (where the ``high-enough'' quantifier on $\lambda$ can depend on $\mu$).
\end{defin}

We can ask whether no long splitting chains (or any other reasonable definition of superstability) implies that $\K$ does not have the tree property, or at least obtain many models from the tree property as in \cite{grsh238}. This is conjectured in \cite{sh394} (see the remark after Claim 5.5 there).

As for the D-rank in \ref{fo-superstab}(5), perhaps a simpler analog would be the $U$-rank defined in terms of $(<\kappa)$-satisfiability in \cite[7.2]{bg-v11-toappear} (another candidate for a rank is Lieberman's $R$-rank, see \cite{liebermanrank}). By \cite[7.9]{bg-v11-toappear}, no long splitting chains implies that the $U$-rank is bounded but we do not know how to prove the converse. Perhaps it is possible to show that $U[p] = \infty$ implies the tree property.

\bibliographystyle{amsalpha}
\bibliography{superstab-defs}

\end{document}